\documentclass{article}
\usepackage{amsthm}
\usepackage{amssymb}
\usepackage{amsfonts}
\begin{document}

\newtheorem{theorem}{Theorem}[section]
\newtheorem{definition}{Definition}[section]
\newtheorem{corollary}[theorem]{Corollary}
\newtheorem{lemma}[theorem]{Lemma}
\newtheorem{proposition}[theorem]{Proposition}
\newtheorem{example}[theorem]{Example}
\newtheorem{remark}[theorem]{Remark}
\newtheorem{claim}[theorem]{Claim}

\font\sixbb=msbm6
\font\eightbb=msbm8
\font\twelvebb=msbm10 scaled 1095
\newfam\bbfam
\textfont\bbfam=\twelvebb \scriptfont\bbfam=\eightbb
                           \scriptscriptfont\bbfam=\sixbb

\newcommand{\tr}{{\rm tr \,}}
\newcommand{\linspan}{{\rm span\,}}
\newcommand{\rank}{{\rm rank\,}}
\newcommand{\diag}{{\rm Diag\,}}
\newcommand{\Image}{{\rm Im\,}}
\newcommand{\Ker}{{\rm Ker\,}}

\def\bb{\fam\bbfam\twelvebb}
\newcommand{\enp}{\begin{flushright} $\Box$ \end{flushright}}
\def\cD{{\mathcal{D}}}
\def\N{\bb N}
\def\R{\bb R}
\def\Z{\bb Z}
\def\Q{\bb Q}

\title{Approximate isometries of Hilbert spaces 
\thanks{This research was supported by grants P1-0288 and N1-0368 from ARIS, Slovenia.}}
\author{
Peter \v Semrl\footnote{Institute of Mathematics, Physics, and Mechanics, Jadranska 19, SI-1000 Ljubljana, Slovenia; Faculty of Mathematics and Physics, University of Ljubljana,
        Jadranska 19, SI-1000 Ljubljana, Slovenia,  peter.semrl@fmf.uni-lj.si}
        }

\date{}
\maketitle

\begin{abstract}
We improve the Hyers - Ulam stability result for isometries of real Hilbert spaces by removing the surjectivity assumption.
\end{abstract}
\maketitle




\section{Introduction and statement of the main result}

Throughout the paper all Hilbert spaces will be over the real field.
Let $H$ and $K$ be Hilbert spaces and $\varepsilon > 0$.
A map $f: H \to K$ is called an $\varepsilon$-isometry if
$$
| \, \| f(x) - f(y) \| - \| x - y \| \, | \le \varepsilon
$$
for all $x,y \in H$. Note that when studying approximate isometries there is no loss of generality in assuming that they map the origin of $H$ to the origin of $K$. Indeed, a map $f: H \to K$ is an $\varepsilon$-isometry
if and only if the map $x \mapsto f(x) - f(0)$, $x \in H$, is an  $\varepsilon$-isometry from $H$ to $K$.  
In 1945 Hyers and Ulam proved the stability result for isometries of Hilbert spaces.

\begin{theorem}\label{maje}\cite[Theorem 4]{HyU}
Let $H$ be a Hilbert space and $f : H \to H$ a surjective $\varepsilon$-isometry satisfying $f(0) = 0$. Then the map $U : H \to H$ defined by
$$
Ux = \lim_{n \to \infty} {1 \over 2^n} f\left( 2^n x \right)
$$
is a linear bijective isometry (an orthogonal operator) and we have
\begin{equation}\label{ten}
\| f(x) - Ux \| \le 10 \varepsilon
\end{equation}
for every $x \in H$.
\end{theorem}

Hyers and Ulam gave an example showing that the surjectivity assumption is indispensable in their theorem.
The above result was the beginning of a quite intensive study of approximate isometries not only on Hilbert spaces but on more general real Banach spaces, see for example
\cite{Bou1, Bou2, Bou3, Bou4, Gev, Gru, HUl}. It took until 1995 for finding the optimal bound in (\ref{ten}) which is $2 \varepsilon$ \cite[Main Theorem]{OmS}. It should be mentioned that in this result the surjectivity assumption can be replaced by the weaker approximate surjectivity condition \cite{SeV}. And in the special case when $\dim H = \dim K < \infty$ the stability result holds true even in the absence of the surjectivity assumption \cite{Dil}. Some more details including the generalizations to arbitrary real Banach spaces can be found in \cite[Chapter 15]{BeL}.

Qian \cite{Qia} was the first who considered approximate isometries on arbitrary Bancah spaces in the absence of the surjectivity assumption. 
Let us recall the famous Mazur-Ulam theorem \cite{MaU} stating that every surjective isometry $f$ between two real Banach spaces satisfying $f(0)=0$ is linear.
A nonsurjective extension of this classical theorem was obtained by Figiel (see \cite{Fig} or \cite[14.2]{BeL}). He proved that for any isometry $f: X \to Y$ between two real Banach spaces
with $f(0) = 0$ there is a linear operator $T$ of norm one mapping the closure of the linear span of the range of $f$ onto $X$ such that the product $Tf$ is the identity on $X$.
As observed by Qian the approximate version of Figiel's theorem does not hold for general Banach spaces. 
We will formulate here a result from \cite{SeV} that links Figiel's theorem with Theorem \ref{maje}.

\begin{theorem}\label{zjus}\cite[Theorem 2.6]{SeV} 
Let $H,K$ be Hilbert spaces and $f : H \to K$ an $\varepsilon$-isometry with $f(0) = 0$. Then
there is a continuous linear operator $T : K \to H$ with $\| T \| = 1$ such that
\begin{equation}\label{prd}
\| T f(x) - x \| \le 2 \varepsilon
\end{equation}
for every $x\in H$.
\end{theorem}

For similar results we refer to \cite{CCTZ, ChD, CDZ, ChZ, Qia}.

At first glance Theorem \ref{zjus} looks as the optimal nonsurjective extension of Theorem \ref{maje}. 
But after thinking for a while it seems natural to ask if we can get something more. Namely, the geometry of Hilbert spaces is much simpler
than in general Banach spaces. In particular, every isometry (not necessarily surjective) acting between Hilbert spaces $H$ and $K$ and sending the origin of $H$ to the origin of $K$ is linear. Indeed, let $x,y \in H$. Then the only metric midpoint between $x$ and $y$ is the algebraic midpoint ${ x + y \over 2}$, and therefore every isometry $f : H \to K$
satisfies
$$
f \left( { x + y \over 2 } \right) = {1 \over 2} (f(x) + f(y))
$$
for all pairs of vectors $x,y$. It follows easily that $f$ must be linear.

Hence, when specializing to real Hilbert spaces we do not
need the surjectivity assumption in the Mazur-Ulam theorem. But on the other hand, already Hyers and Ulam observed that the surjectivity assumption is an essential assumption in Theorem \ref{maje}.
Thus, the situation with approximate isometries is definitely more involved than with isometries. 

Nevertheless, one may argue that when dealing with approximate isometries of Hilbert spaces it is not natural to formulate the stability result in the form suggested by Figiel's theorem  as in (\ref{prd}). 
We will present a result that explains the stability of non-surjective approximate isometries acting on Hilbert spaces much better than Theorem \ref{zjus}. Roughly speaking we will show that there exists a linear isometry $U : H \to K$ approximating $f$ in the sense that when decomposing $K$ into the orthogonal direct sum $K = {\rm Im}\, U \oplus ( {\rm Im}\, U )^\perp$ and when restricting our attention to the first summand we have the same situation as in Theorem \ref{maje} with the optimal bound $2\varepsilon$ instead of $10\varepsilon$, while the projection of $f$ on the second summand is small. Here, ${\rm Im}\, U$
denotes the image of $U$.

\begin{theorem}\label{tooe}
Let $H,K$ be Hilbert spaces and $f : H \to K$ an $\varepsilon$-isometry with $f(0) = 0$. Then
there exists a linear isometry $U : H \to K$ such that 
\begin{equation}\label{uzq}
\| P f(x) - Ux \| \le 2 \varepsilon
\end{equation}
and
\begin{equation}\label{zzzt}
\| (I-P) f(x) \| \le \sqrt{6 \varepsilon \| x \|  + \varepsilon^2 }
\end{equation}
for every $x \in H$. Here, $P : K \to K$ is the orthogonal projection of $K$ onto the image of $U$.
\end{theorem}

\noindent
Clearly, this is an improvement of Theorem \ref{zjus}. Indeed, let $U^{-1}$ denote the inverse of a bijective linear isometry $U : H \to {\rm Im}\, U$. 
The inequality (\ref{uzq})
yields $\| U^{-1} (P f(x) - Ux) \| \le 2 \varepsilon$, $x\in H$, which further implies $\| Tf(x) - x \| \le 2 \varepsilon$, $x\in H$,
where $T = U^{-1}P$ is a bounded linear operator of norm one. Let us write $f$ as $f = Pf + (I-P)f$. The new ingredient is the estimate (\ref{zzzt}) showing that the orthogonal projection of $f$ onto the orthogonal complement of the image of $U$ is small.

In the next section we will discuss the optimality of the bounds given in (\ref{uzq}) and (\ref{zzzt}). The last section will be devoted to the proof of Theorem \ref{tooe}.

\section{The optimality of the main theorem}

The inequality (\ref{uzq}) is known to be sharp even in the special case when $f$ is surjective - in this case $U$ is bijective and $P$ is the identity on $K$.

Next, we turn to the optimality of (\ref{zzzt}). On the right hand side of the inequality we have the estimate of the form $\sqrt{ A \| x \| + B }$ and we would like to know what are the optimal values
of the constants $A$ and $B$.
Let ${\R }^2$ be equipped with the Euclidean norm. It is easy to check that the map
$f : {\R } \to {\R }^2$ defined by
$$
f(t) = \left\{ \begin{matrix} { (t,0) & : & t \le 0 \cr (t, \sqrt{ 2 \varepsilon t}) & : & t \ge 0 } \end{matrix} \right.
$$
is an $\varepsilon$-isometry. Every linear isometry $W : {\R } \to {\R }^2$ is of the form $Wt = t (p,q)$, $t \in \R$, where $(p,q)$ is some vector of norm one. If $(p,q) \not= ( \pm 1,0)$
and $P : {\R }^2 \to {\R }^2$ is the orthogonal projection onto the image of $W$ then obviously there exists a negative real number $s$ such that
$$
\| P f(s) \| = \| P (s,0) \| < |s| - 2 \varepsilon,
$$
and therefore,
$$
\| P f(s) - Ws \| \ge \| Ws \| - \| P f(s) \| > |s| - (|s| - 2 \varepsilon) = 2 \varepsilon.
$$
When $(p,q) = (-1,0)$  we have $\| P f(s) - Ws \| = 2|s| > 2 \varepsilon$ for every negative $s\in \mathbb{R}$ with $|s| > \varepsilon$.
 
Hence, the isometry $U$ appearing in the conclusion of Theorem \ref{tooe} for our particular $\varepsilon$-isometry $f$ is defined by $Ut = (t,0)$, $t \in \R$. If we denote by $P$ the orthogonal projection onto the image of
$U$ then $\| (I-P) f(t) \| = \| (0, \sqrt{ 2 \varepsilon t}) \| = \sqrt{ 2 \varepsilon t}$, $t >0$, showing that the best possible value of $A$ is no smaller than $2\varepsilon$. We were unable
to find the sharp value of $A \in [2 \varepsilon ,6\varepsilon]$.

The next example shows that the value $B= \varepsilon^2$ in (\ref{zzzt}) is optimal. Let $\delta >0$.
It is trivial to verify that the map
$f : {\R } \to {\R }^2$ defined by
$$
f(t) = \left\{ \begin{matrix} { (t,0) & : & t \not= \delta \cr (\delta , \varepsilon) & : & t = \delta} \end{matrix} \right.
$$
is an $\varepsilon$-isometry and as above we can easily see
that the unique isometry $U$ appearing in the conclusion of the above theorem for our particular $\varepsilon$-isometry $f$ is defined by $Ut = (t,0)$, $t \in \R$. By our theorem we know that
$$
\| (I-P) f(t) \| \le \sqrt{ A | t | + B }
$$
for some positive constants $A \le 6 \varepsilon $ and $B \le \varepsilon^2$ and we want to verify that $B$ cannot be smaller than $\varepsilon^2$. If we insert $t= \delta$ in the above inequality we arrive at
$$
\varepsilon  \le \sqrt{ A \delta + B } .
$$
This has to be true for every positive $\delta$ and sending $\delta$ to zero we easily conclude that $B$ must be equal to $\varepsilon^2$.

\section{Proof of the stability result for nonsurjective approximate isometries of Hilbert spaces}

In this section we will prove Theorem \ref{tooe}.

\begin{proof} [Proof of Theorem \ref{tooe}]
We first recall Theorem 1 from the Hyers-Ulam paper \cite{HyU}:
Let $H,K$ be Hilbert spaces and $f : H \to K$ an $\varepsilon$-isometry satisfying $f(0) = 0$. Then the limit
$$
Ux = \lim_{n \to \infty} {1 \over 2^n} f\left( 2^n x \right)
$$
exists for every $x\in H$. The map $x \to Ux$ is an isometry.

In fact, Hyers and Ulam have considered only $\varepsilon$-isometries mapping $H$ into itself. But exactly the same proof works also
for $\varepsilon$-isometries mapping $H$ into some other Hilbert space $K$.

Let $U$ be as above and $T$ as in Theorem \ref{zjus}. Then for every $z \in H$ and every positive integer $n$ we have
$$
\| T f(2^n z) - 2^n z \| \le 2 \varepsilon.
$$
Dividing by $2^n$, sending $n$ to infinity, and using linearity and continuity of $T$ we conclude that
$$
(TU)z = z
$$
for every $z \in H$.

Now we apply the fact that $U$ is linear. In particular, if we denote the image of $U$ by $K_1$, then $K_1$ is a closed subspace of $K$. Let $K_2$ denote its orthogonal
complement. By $U^{-1} : K_1 \to H$ we will denote the inverse of the linear isometry $U$ considered as a bijective isometry from $H$ onto $K_1$.

If $x+y$, $x \in K_1$, $y\in K_2$, is an arbitrary vector in $K$, then because of $(TU)z = z$, $z \in H$, we have
$$
Tx = U^{-1}x.
$$
In the next step we will verify that $Ty = 0$.

Assume on the contrary that $Tu = w \not= 0$ for some unit vector $u \in K_2$. 
Denote $a = \| w \| >0$. For every positive real $t$ we have
$$
T \left( t (Uw) + u \right) = (t+1) w.
$$
If we choose a positive real number $t$ such that
$$
t > { 1 - a^2 \over 2a^2} ,
$$
then a straightforward calculation shows that
$$
1 < { (t+1)a \over \sqrt {t^2 a^2 + 1} } = { \left\| T \left( t (Uw) + u \right)  \right\| \over \|  t (Uw) + u \| },
$$
contradicting the fact that $\| T \| = 1$. 

Hence,  for every $x+y \in K_1 \oplus K_2 = K$ we have 
\begin{equation}\label{psmkon}
T(x+y) = U^{-1}x.
\end{equation}

Let $z\in H$ be any vector. Then there are unique vectors $f_1 (z) \in K_1$ and $k(z) \in K_2$ such that
$$
f(z) = f_1 (z) + k(z).
$$
We define a map $h : H \to K_1$ by $h(z) = f_1 (z) -Uz$, $z \in H$.
From (\ref{prd}) we conclude that
$$
\| T ( Uz + h(z) + k(z) ) - z \| \le 2 \varepsilon
$$
which together with (\ref{psmkon}) yields that $\|  (z + U^{-1} h(z)) - z \| \le 2 \varepsilon$. Applying the fact that $U$ is an isometry
we arrive at
\begin{equation}\label{mvcc}
\| h(z) \| \le 2 \varepsilon
\end{equation}
for every $z \in H$.

We denote by $P : K \to K$ the orthogonal projection onto $K_1$.
Then clearly,
$$
Pf = f_1 \ \ \ {\rm and} \ \ \ (I - P) f = k.
$$
Thus, (\ref{uzq}) follows directly from (\ref{mvcc}), and we only need to check that
$$
\| k(x) \| \le  \sqrt{6 \varepsilon \| x \| + \varepsilon^2}
$$
for every $x\in H$. Because $I-P$ is the orthogonal projection onto the orthogonal complement of ${\rm Im}\, U$ it is enough
to show that
\begin{equation}\label{mile}
\| f(x) - Ux \| \le  \sqrt{6 \varepsilon \| x \| + \varepsilon^2}
\end{equation}
for every $x \in H$.

Let $x\in H$ be any nonzero vector and $k$ any positive integer. Set $\| x \| = r$. Since $f(0) = 0$ and $f$ is an $\varepsilon$-isometry we have
\begin{equation}\label{mil1}
| \, \| f(x) \| - r  \, | \le \varepsilon,
\end{equation}
\begin{equation}\label{mil2}
| \, \| f(kx) - f(x) \| -  (k-1)r  \, | \le \varepsilon,
\end{equation}
and
\begin{equation}\label{mil3}
| \, \| f(kx) \| - kr \, | \le \varepsilon .
\end{equation}
We denote
$$
B_1 = \{ y \in K \, : \, \| y \| \le r + \varepsilon \}
$$
and
$$
B_2 = \{ y \in K \, : \, \| y - f(kx) \| \le (k-1) r + \varepsilon \}.
$$
From (\ref{mil1}) and (\ref{mil2}) we get that  $f(x) \in B_1 \cap B_2$.

For an arbitrary vector $y  \in B_1 \cap B_2$ we have
$$
\| y - f(kx) \|^2 = \| y \|^2 + \| f(kx) \|^2 - 2 \langle y , f(kx) \rangle \le ( (k-1)r + \varepsilon)^2
$$
and therefore,
$$
\left\| y - {1 \over k} f(kx) \right\|^2 = \| y\|^2 + {1 \over k^2} \| f(kx) \|^2 - {2 \over k} \langle y, f(kx) \rangle  
$$
$$
= {1 \over k} ( \| y \|^2 + \| f(kx) \|^2 - 2 \langle y , f(kx) \rangle ) + {k-1 \over k} \| y \|^2 + {1-k \over k^2} \| f(kx) \|^2 
$$
$$
\le {1 \over k} ( (k-1)r + \varepsilon)^2 + {k-1 \over k} \| y \|^2 - {k-1 \over k^2} \| f(kx) \|^2 .
$$
Assume that the positive integer $k$ satisfies  $kr-\varepsilon > 0$.
Applying $\| y \| \le r+ \varepsilon$ and (\ref{mil3}) which yields that $\| f(kx) \| \ge kr - \varepsilon$ we obtain after a straightforward computation that
$$
\left\| y - {1 \over k} f(kx) \right\|^2 \le 6\varepsilon r  \left( 1 - {1 \over k}  \right)
+ \varepsilon ^2 \left( 1 - {1 \over k} + {1 \over k^2} \right).
$$
Putting into the last inequality $y= f(x) \in B_1 \cap B_2$, $k= 2^n$, and sending $n$ to infinity we get the desired inequaltiy (\ref{mile}).
\end{proof}

\end{document}